\documentclass{scrartcl}
\usepackage[T1]{fontenc}
\usepackage[utf8]{inputenc}
\usepackage{hyperref}

\usepackage{bbm}
\usepackage{amsmath, amsthm, amssymb}
\usepackage{mathrsfs}
\usepackage{url}
\usepackage{amssymb,amsfonts}
\usepackage[usenames,dvipsnames,table]{xcolor}
\usepackage[linecolor=white,backgroundcolor=white,bordercolor=white,textsize=tiny]{todonotes}

\usepackage{graphicx}
\usepackage{enumitem}

\usepackage[style=alphabetic, sorting=ynt, backref=true]{biblatex}
\newtheorem{theorem}{Theorem}
\newtheorem*{theorem*}{Theorem}
\theoremstyle{plain}

\newtheorem{corollary}[theorem]{Corollary}

\newtheorem{lemma}[theorem]{Lemma}

\theoremstyle{definition} 
\newtheorem{definition}[theorem]{Definition}
\newtheorem{remark}[theorem]{Remark}

\newcommand{\R}{\mathbb{R}}

\newcommand{\E}{\mathbb{E}}

\newcommand{\mute}[1]{}

\let\todon\todo
\renewcommand{\todo}[1]{\todon{\color{red}#1}}

\newcommand{\ds}{d} 
\newcommand{\TC}{T((\R^\ds))} 

\newcommand\proj{\operatorname{proj}}

\begin{document}

\title{The expected signature of Brownian motion stopped on the boundary of a circle has finite radius of convergence}

\author{Horatio Boedihardjo\thanks{University of Reading. HB gratefully acknowledges EPSRC's support (EP/R008205/1).}, Joscha Diehl\thanks{Universit\"at Greifswald},
Marc Mezzarobba\thanks{Sorbonne Universit\'e, CNRS, Laboratoire d'informatique de Paris 6, LIP6, F-75005 Paris, France.
MM~was supported in part by ANR grant ANR-14-CE25-0018-01 (FastRelax).},
Hao Ni \thanks{University College London. HN acknowledges the support
by the EPSRC under the program grant EP/S026347/1 and by the Alan Turing Institute under the EPSRC grant EP/N510129/1. } }
\maketitle

\begin{abstract}
  The expected signature is an analogue of the Laplace transform for probability measures on rough
  paths. A key question in the area has been to identify a general condition to ensure that the expected signature uniquely determines the measures. A sufficient condition has recently been given by Chevyrev and Lyons, and requires a strong upper bound on the expected signature. While the upper bound was verified for many well-known processes up to a deterministic time, it was not known whether the required bound holds for random time. In fact even the simplest case of Brownian motion up to the exit time of a planar disc was open.
  For this particular case we answer this question using a suitable hyperbolic projection of the expected signature.
  The projection satisfies a three-dimensional system of linear PDEs, which (surprisingly) can be solved explicitly,
  and which allows us to show that the upper bound on the expected signature is \emph{not} satisfied.
\end{abstract}

\tableofcontents

\section{Introduction}

Let a probability measure on a subset of the real line have moments of all orders.
Under which conditions do these moments pin down the probability measure uniquely?
This is the well-studied \emph{moment problem}. When the subset is compact,
the answer is always affirmative. In the noncompact case uniqueness is more delicate (see \cite{Sch2017}).

In stochastic analysis one is usually concerned with measures on
some space of paths, the prime example being \emph{Wiener measure} on the space
of continuous functions.
It turns out that for many purposes a good replacement for ``monomials''
in this setting are the iterated integrals of paths.
The collection of all of these integrals is called the \emph{iterated-integrals signature}.

For smooth paths $X: [0,T] \to \R^d$ (and with respect to a time horizon $T>0$), it is defined, using Riemann-Stieltjes integration, as
\begin{align}
  \begin{split}
  \label{eq:signature}
  S(X)_{0,T}
  &:=
  1 + \int_0^T dX_s + \int_0^T \int_0^{r_2} dX_{r_1} \otimes dX_{r_2}\\
  &\qquad
  + \int_0^T \int_0^{r_3} \int_0^{r_2} dX_{r_1} \otimes dX_{r_2} \otimes dX_{r_3} + \dots
  \in \TC := \prod_{n=0}^\infty (\R^d)^{\otimes n}.
  \end{split}
\end{align}
It is well-known (see \cite{BGLY2016} and references therein) that
\begin{itemize}
  \item $S(X)_{0,T} \in G$, where $G \subset \TC$ is the group of \emph{grouplike elements},
  \item $S(X)_{0,T}$ completely characterizes the path $X$ up to reparametrization and up to tree-likeness.
\end{itemize}

Let $X: \Omega \times [0,T] \to \R^d$ now be a stochastic process.
For fixed $\omega \in \Omega$, $t \mapsto X_t(\omega)$ is usually \emph{not} smooth,
so that we have to assume that the stochastic process posesses a ``reasonable'' integration theory.
In particular assume that integrals of the form $\int g(X_s) dX_s$,
exist for a large class of functions $g \in C(\R^d, L(\R^d, \R^n))$ and that the fundamental theorem of calculus holds,
\begin{align*}
  f(X_t) = f(X_0) + \sum_{i=1}^d \int_0^t \partial_{x_i} f(X_s) dX_s^i.
\end{align*}
An important example is Brownian motion with Stratonovich integration.
Other examples include: the Young integration against a fractional Brownian motion
with Hurst parameter strictly larger than $1/2$, \dots
(see \cite{FV10}).

The iterated-integrals signature $S(X)_{0,T}$ defined by the expression
\eqref{eq:signature} --- now, using the given integration theory --- is then a random variable. Let us assume that we can
take its expectation level-by-level, i.e.\ for all $n \ge 1$ (we postpone the
discussion of the choice of norm on $(\R^d)^{\otimes n}$ to later)
\begin{align*}
  \E\left[ \left\|\int_0^T \int_0^{r_n} \dots \int_0^{r_2} dX_{r_1} \otimes \dots \otimes dX_{r_n}\right\| \right] < +\infty.
\end{align*}
We can then define expected signature level-by-level
\newcommand\expSig{\mathsf{ExpSig}}
\begin{align}
  \expSig(X)_T
  &:=
  \E_{\mathsf{p}}[ S(X)_{0,T} ] \notag \\
  &:=
  \sum_{n=0}^\infty \E\left[ \int_0^T \int_0^{r_n} \dots \int_0^{r_2} dX_{r_1} \otimes \dots \otimes dX_{r_n} \right]
  \in
  \TC.
  \label{eq:expSig}
\end{align}
where $\E_{\mathsf{p}}$ denotes the expectation level-by-level ($\mathsf{p}$ for ``projective'') of a $G$-valued random variable.
The question arises:
\begin{center}
  Does $\expSig(X)$ completely characterize the law of $X$?
\end{center}
As we have seen above, the computation of $S(X)_{0,T}$ already incurs
a loss of information: the parametrization of $X$ and any tree-like parts are lost.
The relevant question is hence
\begin{center}
  Does $\expSig(X)_T$ completely characterize the law of $X$, up to parametrization and tree-likeness?
\end{center}
Since this formulation is a bit awkard, and since the (deterministic) step $S(X) \mapsto X$ is completely
understood, we can instead focus on
\begin{center}
  Does $\expSig(X)_T$ completely characterize the law of $S(X)_{0,T}$ on $G$?
\end{center}

A sufficient condition for this to be the case is given in \cite{ChevyrevLyons}:
if $\expSig(X)_T$ has infinite radius of convergence, that is
\begin{align}\label{eq:InfiniteROC}
  \sum_{n \ge 0} \left\|\proj_n \expSig(X)_T\right\| \lambda^n < +\infty,
\end{align}
for all $\lambda > 0$ then the law of $S(X)_{0,T}$ on $G$ is the unique
law with this (projective) expected value.
Here $\proj_n: \TC \to (\R^d)^{\otimes n}$ denotes projection onto tensors of length $n$.

Let us give two examples.
Let $\mu$ be a probability measure on $\R$ having all moments and define
\begin{align*}
  a_n := \int x^n \mu(dx).
\end{align*}
Consider the stochastic process $X_t := t Z$, where $Z$ is distributed according to $\mu$.
Since $X$~is smooth, its signature is well-defined and actually has the simple form
\begin{align*}
  S(X)_{0,T} = 1 + T Z + \frac{T^2}{2!} Z^2 + \frac{T^3}{3!} Z^3 + \dots
\end{align*}
Then
\begin{align*}
  \expSig(X)_T = 1 + T a_1 + \frac{T^2}{2!} a_2 + \frac{T^3}{3!} a_3 + \dots,
\end{align*}
and a sufficient condition for
$\sum_n a_n T^n \lambda^n/n!$
to have infinite radius of convergence is $|a_n| \le C^n$, for some $C > 0$.%
\footnote{The condition $|a_n| \le C^n$ is of course more than enough in the classical moment problem
to have uniqueness for the law $\mu$ on $\R$ (\cite[Example X.6.4]{SR75}).}
Then \cite[Proposition 6.1]{ChevyrevLyons} applies, and the law
of $S(X)_{0,T}$ on $G$ is uniquely determined by these moments.

Consider now the expected signature of a standard Brownian motion $B$ calculated up to some \emph{fixed} time $T > 0$.
It is known (see for example \cite[Proposition 4.10]{LV2004}) that
\begin{align*}
  \expSig( B )_T
  =
  \exp\left( \frac{T}{2} \sum_{i=1}^d e_i \otimes e_i \right).
\end{align*}
It follows that
\begin{equation*}
  \left\| \proj_{2n} \expSig( B ) \right\|
  =
  \left\| \frac{T^n}{2^n n!} \left( \sum_{i=1}^d e_i \otimes e_i \right)^n \right\|
  \le
  \frac{d^n T^n}{2^n n!},
\end{equation*}
and hence
\begin{align*}
  \sum \|\proj_n \expSig(B)_T \| \lambda^n < +\infty,
\end{align*}
for any $\lambda > 0$.
Again, by \cite[Proposition 6.1]{ChevyrevLyons}, the law of $S(B)_{0,T}$ is uniquely determined by $\expSig(B)_T$.

The decay rate of expected signature of the stochastic process up to the exit time from a bounded domain is a very challenging problem, even for the simple Brownian motion case. The literature on the decay rate of the expected signature focuses on the case for the fixed time interval, e.g. \cite{Passeggeri2016} and \cite{FrizRiedel2014}, which heavily relies on the Gaussian tail of the increment. However, the increment of stopped processes would violate this assumption. The question on the finiteness of the convergence radius of the expected signature of the Brownian motion was firstly proposed in \cite{LyonsNi} in 2015. It has remained open until our paper provides the first negative example, i.e. a 2-dimensional Brownian motion up to the unit disk, which implies the lower bound of the decay rate of the expected signature in this case. In \cite[Theorem 3.6]{LyonsNi}, using Sobolev estimates, under certain smoothness and boundedness condition of the domain, geometric upper bounds for the decay rate of the expected signature of stopped Brownian is established (see also \cite[Example 6.20]{ChevyrevLyons} for a probabilistic approach). It is even more challenging to establish a non-trivial lower bound of the decay rate of the expected signature in this case. Our work may shed some light on how to use the PDE approach to derive the lower bound of the decay rate of the stopped diffusion processes. 

Concretely, we consider the Brownian motion $B^z$ in $\R^2$ started at some point $z$ in the unit circle
$\mathbb D := \{ z \in \R^2 : |z| \le 1 \}$,
and stopped at hitting the boundary, that is
\begin{align}
  \label{eq:tau}
  \tau := \inf \left\{ t \ge 0 : |B^z_t| \in \partial \mathbb D \right\}.
\end{align}
In the notation introduced above, we are interested in
\begin{align*}
  \Phi(z) := \expSig( X^z )_\infty,
\end{align*}
where $X^z_t := B^z_{t \wedge \tau}$.
In \cite{LyonsNi} it was shown that
for every $n\in\mathbb{N}$
and $n\geq2$, the $n$th term of $\Phi$ satisfies the following PDE:
\begin{equation}
  \Delta\left(\proj_{n}\left(\Phi(z)\right)\right)
=  -2\sum_{i=1}^{d}e_{i}\otimes\frac{\partial\proj_{n-1}\left(\Phi(z)\right)}{\partial z_{i}}-\left(\sum_{i=1}^{d}e_{i}\otimes e_{i}\right)\otimes\proj_{n-2}\left(\Phi(z)\right),\label{eq:LyonsNiPDE}
\end{equation}
with the boundary condition that for each $|z|=1$, 
\[
\proj_{n}\left(\Phi(z)\right)=\begin{cases}
0, & \text{if}\,n\geq1\\
1, & \text{if }n=0.
\end{cases}
\]
Additionally, one has $\proj_{0}\left(\Phi(z)\right)=1$ and $\proj_{1}\left(\Phi(z)\right)=0$ for all $z \in \mathbb D$.
Using this, they were able to obtain the bound
$\left\|\proj_{n}\left(\Phi(x)\right)\right\|\leq C^{n}$
for some $C > 0$
(\cite[Theorem 3.6]{LyonsNi}).
This bound is \emph{not} enough to decide whether the radius of convergence for $\expSig(X^z)_\infty$ is infinite or not,
but it is enough to deduce that $\expSig(X^z)_\infty$ has radius of converge strictly larger than $0$.
In this work we show that the radius of convergence is indeed finite.


Recall from \cite[Proposition 6.1]{ChevyrevLyons} that
if $A,B$ are $G$-valued random variables such that
$\E_{\mathsf p}[ A ] = \E_{\mathsf p}[ B ]$ and 
$\E_{\mathsf p}[ A ]$ has an infinite radius of convergence, then
$A\overset{\mathcal{D}}{=}B$.
Our main theorem, proven in Section \ref{sec:concluding}, is the following.

\begin{theorem*}
  The expected signature $\Phi(0) = \expSig(X^0)_\infty = \E_{\mathsf p}[ S(X^0)_{0,\infty} ]$ of a two-dimensional Brownian motion stopped upon exiting the unit disk has a \emph{finite} radius of convergence.
\end{theorem*}

The condition of $\E_{\mathsf p}[ A ]$ having an infinite radius of convergence
is equivalent to $\E_{\mathsf p}[ A ]$ lying in $E$.
Here $E$ is defined as the closure of $T((\mathbb{R}^2))$ under the
coarsest topology such that for all normed algebras $\mathcal A$ and all $M \in L(\mathbb{R}^2, \mathcal A)$,
the extension
$M: T((\mathbb{R}^2)) \to \mathcal A$
is continuous.
Recall that for $M\in L\left(\mathbb{R}^2,A\right)$, we may define $M$ firstly
as a map on the $k$-times algebraic tensor product $(\mathbb{R}^2)^{\otimes_{a}k}$,
by the relation
\begin{equation}
  M\left(v_{1}\otimes\cdots\otimes v_{k}\right)=M\left(v_{1}\right)\cdots M\left(v_{k}\right),\label{eq:LinearMapExtension}
\end{equation}
and then extended it to $T((\mathbb{R}^2))$ by linearity.

We want to show that $\Phi(z)$ does \emph{not} lie in the space $E$.
It is sufficient to show that there exists
$\lambda^{*} \in \mathbb{R}$, and $M\in L\left(\mathbb{R}^{2},M_{3\times3}\left(\mathbb{R}\right)\right)$,
such that $\left(\lambda M\right)\left(\Phi(z)\right)$
diverges as $\lambda$ tends to a finite number $\lambda^{*}$. In
fact, we will choose $M$ to be 
\begin{equation}
  M:\begin{pmatrix}
  x\\
  y
  \end{pmatrix}\mapsto\begin{pmatrix}
  0 & 0 & x\\
  0 & 0 & y\\
  x & y & 0
  \end{pmatrix}.\label{eq:HyperbolicDevelopment}
\end{equation}
Such a map $M$ first appeared in \cite{HamblyLyons} to study the signature of bounded variation paths and is also subsequently used in \cite{LyonsXu}.

We proceed as follows.
In Section~\ref{sec:differentiability},
for $\lambda > 0$ we let $\lambda M$ act on $\Phi(z)$.
For $\lambda$ small enough, we show that resulting linear map in $L(\R^3,\R^3)$,
evaluated at $(0,0,1) \in \R^3$ is smooth in $z$ and solves a certain PDE.
Using rotational invariance of Brownian motion, in Section~\ref{sec:polarDecomposition}
we rewrite said PDE solution in polar coordinates.
In Sections \ref{sec:ODEABC}~and~\ref{sec:solvingABC},
we obtain an explicit solution for the PDE (still, for $\lambda$ small enough) in terms of Bessel functions.
Finally, in Section~\ref{sec:concluding} we show that the solution blows up
as $\lambda \to \tilde \lambda$ for some $\tilde \lambda < +\infty$, proving our main theorem.
The Appendix, Section \ref{sec:appendix}, contains some auxiliary results on PDEs.

\section{Differentiability of the development of expected signature}
\label{sec:differentiability}

We first need two technical lemmas which assert that the development
of the expected signature is twice differentiable, and satisfies the PDE
we expect it to. In the lemma below we will adopt the multi-index notation
\[
\left|\left(\alpha_{1},\alpha_{2}\right)\right|=\left|\alpha_{1}\right|+\left|\alpha_{2}\right|,
\qquad
D^{\left(\alpha_{1},\alpha_{2}\right)}u(z)=\frac{\partial^{\alpha_{1}+\alpha_{2}}u}{\partial z_{1}^{\alpha_{1}}\partial z_{2}^{\alpha_{2}}}(z),
\qquad
\left(\alpha_{1},\alpha_{2}\right)\in\mathbb{N}^{2}.
\]

Observe that
\[
\left(\lambda M\right)\Phi(z)=\sum_{n=0}^{\infty}\lambda^{n}M\proj_{n}\left(\Phi(z)\right).
\]
Let $\Vert \cdot \Vert$ be the projective norm.

\begin{lemma}
\label{lem:DifferentiabilityOfExpectedSignature} The function 
$z\mapsto\proj_{n}\left(\Phi(z)\right)$
is twice continuously differentiable.
There exists a constant
$C>0$ such that for all $n\in\mathbb{N}$, all $z\in\mathbb{D}$ and all $\alpha \in \mathbb{N}^2$ satisfying
$|\alpha|\leq2$, one has the bound
\[
\left\Vert D^{\alpha}\proj_{n}\left(\Phi(z)\right)\right\Vert \leq C^{n}.
\]
Moreover, there exists $\lambda^{*}>0$ such that for all $\lambda<\lambda^{*}$
\[
z\mapsto\sum_{n=0}^{\infty}\lambda^{n}M\proj_{n}\left(\Phi(z)\right)
\]
is twice differentiable in $z$ and if $|\alpha|\leq2$,
then 
\begin{equation*}
  D^{\alpha}\sum_{n=0}^{\infty}\lambda^{n}M\proj_{n}\left(\Phi(z)\right)
=  \sum_{n=0}^{\infty}\lambda^{n}D^{\alpha}M\proj_{n}\left(\Phi(z)\right).
\end{equation*}
\end{lemma}
\begin{proof}
Let $m\in\mathbb{N}$. By Theorem \ref{Phi_theorem} in the appendix,
the function 
$z\mapsto\proj_{n}\left(\Phi(z)\right)$
is twice continuously differentiable (it is in fact infinitely differentiable
on $\mathbb{D}$). By Lemma~\ref{lemma_pde} in the Appendix, there exists $C>0$
such that for all $n\in\mathbb{N}$ 
\begin{equation}
\Vert\proj_{n}\left(\Phi(z)\right)\Vert_{W^{m,2}\left(\mathbb{D}\right)}\leq C^{n},\label{eq:GeometricBound}
\end{equation}
where the norm $\Vert\cdot\Vert_{W^{m,2}\left(\mathbb{D}\right)}$
is the Sobolev norm on the unit disc $\mathbb{D}$ with respect to
the variable $z$, 
\[
\Vert u\Vert_{W^{m,2}\left(\mathbb{D}\right)}=\max_{\left|\alpha\right|=m}\Vert D^{\alpha}u\Vert_{L^{2}\left(\mathbb{D}\right)}.
\]


By Theorem 2.2 in \cite{LyonsNi}, which bounds the values
of a function $u$ in terms of the Sobolev norm of $u$, there is
some constant $\tilde{C}(2)$ such that for all $z\in\mathbb{D}$
and $\left|\alpha\right|\leq2$, 
\begin{align*}
 \left|D^{\alpha}\proj_{n}\left(\Phi(z)\right)\right|
& \leq \tilde{C}(2)\Vert D^{\alpha}\proj_{n}\left(\Phi(z)\right)\Vert_{W^{2,2}\left(\mathbb{D}\right)}\\
& \leq \tilde{C}(2)\Vert\proj_{n}\left(\Phi(z)\right)\Vert_{W^{4,2}\left(\mathbb{D}\right)}\\
& \leq \tilde{C}(2)\,C(4)^{n}.
\end{align*}
Since $M\proj_{n}\left(\Phi(z)\right)$ is a linear image of $\proj_{n}\left(\Phi(z)\right)$,
the function $M\proj_{n}\left(\Phi(z)\right)$ is twice continuously
differentiable in $z$, and moreover, there exists $c>0$ such that
for all $z\in\mathbb{D}$, 
\[
\left|D^{\alpha}M\proj_{n}\left(\Phi(z)\right)\right|\leq c^{n}.
\]

This bound also allows us to deduce that for $\left|\alpha\right|=2$, the series
\[
\sum_{n=0}^{\infty}\lambda^{n}D^{\alpha}M\proj_{n}\left(\Phi(z)\right)
\]
converges uniformly and hence the series 
\[
\sum_{n=0}^{\infty}\lambda^{n} M\proj_{n}\left(\Phi(z)\right)
\]
is twice continuously differentiable and the derivatives can be taken inside the infinite summation.  
\end{proof}

\begin{lemma}
\label{lem:PDEForDevelopment}There exists $\lambda^{*}>0$ such that
if $\lambda<\lambda^{*}$, the function~$F_\lambda$ defined by
\begin{equation} \label{eq:def-F}
F_\lambda(z)
= (\lambda M) \Phi(z)
\begin{pmatrix}0\\ 0\\ 1 \end{pmatrix}
=\sum_{n=0}^{\infty}\lambda^{n}M\proj_{n}\left(\Phi(z)\right)
\begin{pmatrix}0\\ 0\\ 1 \end{pmatrix}
\end{equation}
is twice continuously differentiable on $\mathbb{D}$, and satisfies
\[
\Delta F_\lambda(z)=-2\lambda\sum_{i=1}^{2}Me_{i}\frac{\partial F_\lambda}{\partial z_{i}}(z)-\lambda^{2}\left(\sum_{i=1}^{2}\left(Me_{i}\right)^{2}\right)F_\lambda(z)
\]
with $F_\lambda(z)=(0,0,1)$ for $z\in\partial\mathbb{D}$. Here $(e_1,e_2)$ denotes the canonical basis of $\mathbb{R}^2$. 
\end{lemma}

\begin{proof}
If we apply the linear map $M$ to the PDE (\ref{eq:LyonsNiPDE}),
then we have a matrix-valued PDE
\begin{multline}
 \Delta\left(M\proj_{n}\left(\Phi(z)\right)\right)\\
= -2\sum_{i=1}^{2}Me_{i}\frac{\partial M\proj_{n-1}\left(\Phi(z)\right)}{\partial z_{i}}-\left(\sum_{i=1}^{2}\left(Me_{i}\right)^{2}\right)M\proj_{n-2}\left(\Phi(z)\right),\label{eq:hyperbolicPDE}
\end{multline}
together with the boundary condition 
\begin{align*}
M\proj_{0}\left(\Phi(z)\right) & =I_{3\times3}\\
M\proj_{1}\left(\Phi(z)\right) & =0_{3\times3}.
\end{align*}

We may multiply both sides with $\lambda^{n}$, sum to infinity and
apply to the vector $(0,0,1)$ to get
\begin{align*}
 & \sum_{n=0}^{\infty}\lambda^{n}\Delta\left(M\proj_{n}\left(\Phi(z)\right)\right)\begin{pmatrix}0\\
0\\
1
\end{pmatrix}\\
= & -2\lambda\sum_{i=1}^{2}Me_{i}\sum_{n=0}^{\infty}\lambda^{n}\frac{\partial}{\partial z_{i}}\proj_{n}\left(M\Phi(z)\right)\begin{pmatrix}0\\
0\\
1
\end{pmatrix}-\lambda^{2}\left(\sum_{i=1}^{2}\left(Me_{i}\right)^{2}\right)F_\lambda(z).
\end{align*}
By Lemma \ref{lem:DifferentiabilityOfExpectedSignature}, each of
the infinite sums converges and we may take the derivatives outside
the infinite sum. 
\end{proof}

\section{A polar decomposition for the development}
\label{sec:polarDecomposition}

Let $x=\left(x_{1},x_{2}\right)^{t}\in\mathbb{R}^{2}$. Recall that
\[
M(x)=\begin{pmatrix}
0 & 0 & x_{1}\\
0 & 0 & x_{2}\\
x_{1} & x_{2} & 0
\end{pmatrix}.
\]

We may consider $M(x)$ as a linear endomorphism of $\mathbb{R}^{2}\oplus\mathbb{R}$
mapping $\left(v,\alpha\right)$ to $\left(\alpha x,\left\langle x,v\right\rangle \right)$.

\begin{lemma}
\label{lem:RotatedDevelopment} 
For any linear map $R: \R^2 \to \R^2$, 
\[
M\left(R(x)\right)=\left(R\oplus1\right)M(x)\left(R^{*}\oplus1\right),
\]
where $R^{*}$ is the transpose of $R$.

\end{lemma}
\begin{proof}
Note that
\begin{align*}
  (R\oplus1)M(x)(R^{*}\oplus1)(v,\alpha)
&= (R\oplus1)M(x)(R^{*}v,\alpha)
= (R\oplus1)(\alpha x,\left\langle x,R^{*}v\right\rangle )\\
&= (\alpha R(x),\left\langle x,R^{*}v\right\rangle )
= (\alpha R(x),\left\langle Rx,v\right\rangle )
= M(R(x)). \qedhere
\end{align*}
\end{proof}

In what follows, we will use the notation 
\[
\triangle_{n}\left(0,t\right)=\left\{ \left(t_{1},\ldots,t_{n}\right):0<t_{1}<\cdots<t_{n}<t\right\} .
\]

\begin{corollary}
\label{cor:RotatedExpectedDevelopment} Let $R:\mathbb{R}^2\rightarrow \mathbb{R}^2$ be the rotation map
\[
z \rightarrow \begin{pmatrix}
\cos\theta & -\sin\theta \\
\sin\theta & \cos\theta 
\end{pmatrix} z.
\]
Then
\[
\left(\lambda M\right)\Phi\left(R(z)\right)=\left(R\oplus1\right)\left(\lambda M\right)\Phi(z) \left(R^{*}\oplus1\right).
\]
\end{corollary}
\begin{proof}
Brownian motion $B^{R(z)}$ starting at $R(z)$ has the same distribution
as the rotated Brownian motion $R\left(B^{z}\right)$, where $B^{z}$ starts
from $z$. Let $\circ \mathrm{d}$ denote the Stratonovich differential. Then
\begin{align*}
 \left(\lambda M\right)\Phi\left(R(z)\right)
= & \sum_{n=0}^{\infty} \lambda^n \mathbb{E}^{R(z)}\left[\int_{\triangle_{n}\left(0,\tau_{\mathbb{D}}\right)}M\left(\circ\mathrm{d}B^{z}_{t_{1}}\right)\cdots M\left(\circ\mathrm{d}B^{z}_{t_{n}}\right)\right]\\
= & \sum_{n=0}^{\infty} \lambda^n \mathbb{E}^{z}\left[\int_{\triangle_{n}\left(0,\tau_{\mathbb{D}}\right)}M\left(R\left(\circ\mathrm{d}B^{z}_{t_{1}}\right)\right)\cdots M\left(R\left(\circ\mathrm{d}B^{z}_{t_{n}}\right)\right)\right].
\end{align*}
By Lemma \ref{lem:RotatedDevelopment}
\begin{align*}
 & \int_{\triangle_{n}\left(0,\tau_{\mathbb{D}}\right)}M\left(R\left(\circ\mathrm{d}B_{t_{1}}\right)\right)\cdots M\left(R\left(\circ\mathrm{d}B_{t_{n}}\right)\right)\\
= & \int_{\triangle_{n}\left(0,\tau_{\mathbb{D}}\right)}\left(R\oplus1\right)M\left(\circ\mathrm{d}B_{t_{1}}\right)\left(R^{*}\oplus1\right)\cdots\left(R\oplus1\right)M\left(\circ\mathrm{d}B_{t_{n}}\right)\left(R^{*}\oplus1\right).
\end{align*}
As $R$ is orthogonal, we have 
\begin{align*}
 & \int_{\triangle_{n}\left(0,\tau_{\mathbb{D}}\right)}M\left(R\left(\circ\mathrm{d}B_{t_{1}}\right)\right)\cdots M\left(R\left(\circ\mathrm{d}B_{t_{n}}\right)\right)\\
= & \left(R\oplus1\right) \int_{\triangle_{n}\left(0,\tau_{\mathbb{D}}\right)}M\left(\circ\mathrm{d}B_{t_{1}}\right)\cdots M\left(\circ\mathrm{d}B_{t_{n}}\right) \left(R^{*}\oplus1\right).
\end{align*}
Therefore, 
\begin{align*}
 & \left(\lambda M\right)\Phi\left(R(z)\right)\\
&= \left(R\oplus1\right)\sum_{n=0}^{\infty}\lambda^{n}\mathbb{E}^{z}\int_{\triangle_{n}\left(0,\tau_{\mathbb{D}}\right)}M\left(\circ\mathrm{d}B^{z}_{t_{1}}\right)\cdots M\left(\circ\mathrm{d}B^{z}_{t_{n}}\right)\left(R^{*}\oplus1\right) \\
&= \left(R\oplus1\right)\left(\lambda M\right)\Phi(z)\left(R^{*}\oplus1\right).
\qedhere
\end{align*}
\end{proof}

\begin{corollary}
\label{cor:A and C} Define the functions
$A_{\lambda},B_{\lambda},C_{\lambda}:\left[0,1\right]\rightarrow\mathbb{R}$
by 
\begin{equation}
\begin{pmatrix}
A_{\lambda}(r)\\
B_{\lambda}(r)\\
C_{\lambda}(r)
\end{pmatrix}=F_\lambda(r,0)\label{eq:DefinitionABC}
\end{equation}
where $F_\lambda$ is the function defined by~\eqref{eq:def-F}.
In polar coordinates, the expression of~$F_\lambda$ reads
\[
F(r \cos \theta, r \sin \theta)=\begin{pmatrix}
\cos\theta & -\sin\theta & 0\\
\sin\theta & \cos\theta & 0\\
0 & 0 & 1
\end{pmatrix}\begin{pmatrix}
A_{\lambda}\left(r \right)\\
B_{\lambda}(r)\\
C_{\lambda}(r)
\end{pmatrix}.
\]
Additionally, there exists $\lambda^{*}>0$ such that if $\lambda<\lambda^{*}$, then $A_{\lambda},B_{\lambda},C_{\lambda}$ are twice continuously differentiable
functions in the variable $r$ for all $r\in\left[0,1\right]$. 
\end{corollary}

\begin{proof}
The functions $A_{\lambda},B_{\lambda},C_{\lambda}$ are twice continuously differentiable because $F_\lambda(r,0)$ is twice continuously differentiable by Lemma \ref{lem:PDEForDevelopment}.
Let $R: \mathbb{R}^2 \to \mathbb{R}^2$ be the rotation of angle~$\theta$.
For $z = (r \cos \theta, r \sin \theta) = R(r, 0)$,
the definition of~$F_\lambda$ gives
\begin{equation*}
  F_\lambda(z)
= \left(\lambda M\right)\Phi(z)\begin{pmatrix}
0\\
0\\
1
\end{pmatrix}
= \left(\lambda M\right)\Phi\left(R(r,0)\right)
\begin{pmatrix}
0\\
0\\
1
\end{pmatrix}
\end{equation*}
By Corollary \ref{cor:RotatedExpectedDevelopment}, one has
\[
F_\lambda(z) = (R \oplus 1) (\lambda M) \Phi(R(r,0)) (R^{*} \oplus 1)
\begin{pmatrix}
0\\
0\\
1
\end{pmatrix}
= (R \oplus 1) F_\lambda(r, 0)
\]
where
\[ R \oplus 1 = \begin{pmatrix}
\cos\theta & -\sin\theta & 0\\
\sin\theta & \cos\theta & 0\\
0 & 0 & 1
\end{pmatrix}. \qedhere \]
\end{proof}

\section{ODE for $A_{\lambda},B_{\lambda},C_{\lambda}$}
\label{sec:ODEABC}

\begin{lemma}
\label{lem:EquationForABC}There exists $\lambda^{*}>0$ such that
for $\lambda<\lambda^{*}$, the functions $A_{\lambda},B_{\lambda},C_{\lambda}$ defined in Lemma
\ref{cor:A and C} satisfy
\begin{equation}
\begin{split}r^{2}A_{\lambda}^{\prime\prime}(r)+rA_{\lambda}^{\prime}(r)-A_{\lambda}(r)+\lambda^{2}r^{2}A_{\lambda}(r)+2\lambda r^{2}C_{\lambda}^{\prime}(r) & =0\\
r^{2}B_{\lambda}^{\prime\prime}(r)+B_{\lambda}^{\prime}(r)r-B_{\lambda}(r)+r^{2}\lambda^{2}B_{\lambda}(r) & =0\\
C_{\lambda}^{\prime}(r)+rC_{\lambda}^{\prime\prime}(r)+2\lambda^{2}rC_{\lambda}(r)+2\lambda rA_{\lambda}^{\prime}(r)+2\lambda A_{\lambda}(r) & =0
\end{split}
\label{eq:EquationsForABCRepeat}
\end{equation}
and 
\begin{align*}
A_{\lambda}(0)=0,\,B_{\lambda}(0)=0, & A_{\lambda}(1)=0\\
C_{\lambda}^{\prime}(0)=0,\,B_{\lambda}(1)=0, & C_{\lambda}(1)=1.
\end{align*}
\end{lemma}

\begin{proof}
By Corollary \ref{cor:A and C}, 
\[
F_\lambda\begin{pmatrix}
r\cos\theta\\
r\sin\theta
\end{pmatrix}=\begin{pmatrix}
\cos\theta A_{\lambda}(r)-\sin\theta B_{\lambda}(r)\\
\sin\theta A_{\lambda}(r)+\cos\theta B_{\lambda}(r)\\
C_{\lambda}(r)
\end{pmatrix}.\label{eq:FInTermsABC}
\]

As $A_{\lambda},B_{\lambda},C_{\lambda}$ are twice continuously differentiable for $\lambda<\lambda^{*}$, we may substitute (\ref{eq:FInTermsABC}) for $F_\lambda$ into the equation in Lemma \ref{lem:PDEForDevelopment}, which
gives for $r>0$
\begin{align}
\Delta\begin{pmatrix}
F_\lambda^{(1)}\\
F_\lambda^{(2)}\\
F_\lambda^{(3)}
\end{pmatrix} & =\begin{pmatrix}
\partial_{rr}F_\lambda^{(1)}+\frac{1}{r}\partial_{r}F_\lambda^{(1)}+\frac{1}{r^{2}}\partial_{\theta\theta}F_\lambda^{(1)}\\
\partial_{rr}F_\lambda^{(2)}+\frac{1}{r}\partial_{r}F_\lambda^{(2)}+\frac{1}{r^{2}}\partial_{\theta\theta}F_\lambda^{(2)}\\
\partial_{rr}F_\lambda^{(3)}+\frac{1}{r}\partial_{r}F_\lambda^{(3)}+\frac{1}{r^{2}}\partial_{\theta\theta}F_\lambda^{(3)}
\end{pmatrix}\nonumber \\
 & =\begin{pmatrix}
\cos\theta A_{\lambda}^{\prime\prime}(r)-\sin\theta B_{\lambda}^{\prime\prime}(r)+\frac{1}{r}\left(\cos\theta A_{\lambda}^{\prime}(r)-\sin\theta B_{\lambda}^{\prime}(r)\right)\\
+\frac{1}{r^{2}}\left(-\cos\theta A_{\lambda}(r)+\sin\theta B_{\lambda}(r)\right)\\
\sin\theta A_{\lambda}^{\prime\prime}(r)+\cos\theta B_{\lambda}^{\prime\prime}(r)+\frac{1}{r}\left(\sin\theta A_{\lambda}^{\prime}(r)+\cos\theta B_{\lambda}^{\prime}(r)\right)\\
+\frac{1}{r^{2}}\left(-\sin\theta A_{\lambda}(r)-\cos\theta B_{\lambda}(r)\right)\\
C_{\lambda}^{\prime\prime}(r)+\frac{1}{r}C_{\lambda}^{\prime}(r)
\end{pmatrix}.\label{eq:PolarLaplacian}
\end{align}
Using the identities
\begin{align*}
\frac{\partial}{\partial z_{1}} &=\frac{\partial r}{\partial z_{1}}\frac{\partial}{\partial r}+\frac{\partial\theta}{\partial z_{1}}\frac{\partial}{\partial\theta}
  =\cos\theta\frac{\partial}{\partial r}-\frac{\sin\theta}{r}\frac{\partial}{\partial\theta}, \\
\frac{\partial}{\partial z_{2}} &=\frac{\partial r}{\partial z_{2}}\frac{\partial}{\partial r}+\frac{\partial}{\partial\theta}\frac{\partial\theta}{\partial z_{2}}
  =\sin\theta\frac{\partial}{\partial r}+\frac{\cos\theta}{r}\frac{\partial}{\partial\theta},
\end{align*}
the right hand side of the PDE in Lemma \ref{lem:PDEForDevelopment}
is
\begin{align}
 & \begin{pmatrix}
-\lambda^{2}\cos\theta A_{\lambda}(r)+\lambda^{2}\sin\theta B_{\lambda}(r)-2\lambda\cos\theta C_{\lambda}^{\prime}(r)\\
-\lambda^{2}\sin\theta A_{\lambda}(r)-\lambda^{2}\cos\theta B_{\lambda}(r)-2\lambda\sin\theta C_{\lambda}^{\prime}(r)\\
-2\lambda^{2}C_{\lambda}(r)-2\lambda\left[\cos^{2}\theta A_{\lambda}^{\prime}(r)-\cos\theta\sin\theta B_{\lambda}^{\prime}(r)\right]\\
\dots - 2\lambda\left[\frac{\sin\theta}{r}\left(\sin\theta A_{\lambda}(r)+\cos\theta B_{\lambda}(r)\right)\right]\\
\dots-2\lambda\left[\sin^{2}\theta A_{\lambda}^{\prime}(r)+\sin\theta\cos\theta B_{\lambda}^{\prime}(r)\right]\\
\dots-2\lambda\left[\frac{\cos\theta}{r}\cos\theta A_{\lambda}(r)-\frac{\cos\theta}{r}\sin\theta B_{\lambda}(r)\right]
\end{pmatrix}\nonumber \\
= & \begin{pmatrix}
-\lambda^{2}\cos\theta A_{\lambda}(r)+\lambda^{2}\sin\theta B_{\lambda}(r)-2\lambda\cos\theta C_{\lambda}^{\prime}(r)\\
-\lambda^{2}\sin\theta A_{\lambda}(r)-\lambda^{2}\cos\theta B_{\lambda}(r)-2\lambda\sin\theta C_{\lambda}^{\prime}(r)\\
-2\lambda^{2}C_{\lambda}(r)-2\lambda A_{\lambda}^{\prime}(r)-\frac{2\lambda}{r}A_{\lambda}(r)
\end{pmatrix}.\label{eq:RHSPDE}
\end{align}
Equating (\ref{eq:PolarLaplacian}) and (\ref{eq:RHSPDE}) gives the
first equation as
\begin{align*}
 & \cos\theta A_{\lambda}^{\prime\prime}(r)-\sin\theta B_{\lambda}^{\prime\prime}(r)+\frac{1}{r}\left(\cos\theta A_{\lambda}^{\prime}(r)-\sin\theta B_{\lambda}^{\prime}(r)\right)\\
 & +\frac{1}{r^{2}}\left(-\cos\theta A_{\lambda}(r)+\sin\theta B_{\lambda}(r)\right)\\
= & -\lambda^{2}\cos\theta A_{\lambda}(r)+\lambda^{2}\sin\theta B_{\lambda}(r)-2\lambda\cos\theta C_{\lambda}^{\prime}(r).
\end{align*}
As this holds for all $\theta$, we may equate the coefficients of $\sin \theta$ and $\cos \theta$ to obtain 
\begin{align}
A_{\lambda}^{\prime\prime}(r)+\frac{A_{\lambda}^{\prime}(r)}{r}-\frac{1}{r^{2}}A_{\lambda}(r) & =-\lambda^{2}A_{\lambda}(r)-2\lambda C_{\lambda}^{\prime}(r)\nonumber \\
-B_{\lambda}^{\prime\prime}(r)-\frac{1}{r}B_{\lambda}^{\prime}(r)+\frac{1}{r^{2}}B_{\lambda}(r) & =\lambda^{2}B_{\lambda}(r)\label{eq:ABCsystem1}
\end{align}
and from the second equation, 
\begin{align*}
 & \sin\theta A_{\lambda}^{\prime\prime}(r)+\cos\theta B_{\lambda}^{\prime\prime}(r)+\frac{1}{r}\left(\sin\theta A_{\lambda}^{\prime}(r)+\cos\theta B_{\lambda}^{\prime}(r)\right)\\
 & +\frac{1}{r^{2}}\left(-\sin\theta A_{\lambda}(r)-\cos\theta B_{\lambda}(r)\right)\\
= & -\lambda^{2}\sin\theta A_{\lambda}(r)-\lambda^{2}\cos\theta B_{\lambda}(r)-2\lambda\sin\theta C_{\lambda}^{\prime}(r)
\end{align*}
and therefore, 
\begin{align}
A_{\lambda}^{\prime\prime}(r)+\frac{1}{r}A_{\lambda}^{\prime}(r)-\frac{1}{r^{2}}A_{\lambda}(r) & =-\lambda^{2}A_{\lambda}(r)-2\lambda C_{\lambda}^{\prime}(r)\nonumber \\
B_{\lambda}^{\prime\prime}(r)+\frac{B_{\lambda}^{\prime}(r)}{r}-\frac{B_{\lambda}(r)}{r^{2}} & =-\lambda^{2}B_{\lambda}(r).\label{eq:ABCsystem2}
\end{align}
Combining (\ref{eq:ABCsystem1}) and (\ref{eq:ABCsystem2}) and multiplying
the equations throughout by $r^{2}$, we have 
\begin{equation}
\begin{split}r^{2}A_{\lambda}^{\prime\prime}(r)+rA_{\lambda}^{\prime}(r)-A_{\lambda}(r)+\lambda^{2}r^{2}A_{\lambda}(r)+2\lambda r^{2}C_{\lambda}^{\prime}(r) & =0\\
r^{2}B_{\lambda}^{\prime\prime}(r)+B_{\lambda}^{\prime}(r)r-B_{\lambda}(r)+r^{2}\lambda^{2}B_{\lambda}(r) & =0\\
C_{\lambda}^{\prime}(r)+rC_{\lambda}^{\prime\prime}(r)+2\lambda^{2}rC_{\lambda}(r)+2\lambda rA_{\lambda}^{\prime}(r)+2\lambda A_{\lambda}(r) & =0
\end{split}
\label{eq:EquationsForABC}
\end{equation}
for all $r>0$. By continuity of the second derivatives of $A_{\lambda},B_{\lambda},C_{\lambda}$
(see Lemma \ref{cor:A and C}), the
equations hold for all $r\geq0$. Again using the continuity of the
second derivatives of $A_{\lambda},B_{\lambda},C_{\lambda}$, we may substitute $r=0$ into (\ref{eq:EquationsForABC})
to get 
\begin{equation*}
A_{\lambda}(0) = B_{\lambda}(0) = C_{\lambda}^{\prime}(0) =0.
\end{equation*}
Using the boundary conditions for $z\in\partial\mathbb{D}$ in Lemma \ref{lem:PDEForDevelopment},
we have 
\[
  \begin{pmatrix}
    A_{\lambda}(1) \\ B_{\lambda}(1) \\ C_{\lambda}(1)
  \end{pmatrix}
  =
  F_\lambda(r, 0)
  = 
  \begin{pmatrix}
    0 \\ 0 \\ 1
  \end{pmatrix}.
  \qedhere
\]
\end{proof}

\section{Solving the ODE for $A_{\lambda},B_{\lambda},C_{\lambda}$}

\label{sec:solvingABC}

\begin{lemma} \label{lem:SolvingForABC}
  Let
  \begin{equation} \label{eq:zeta etc.}
    \zeta = \sqrt{\frac{- 1 + i \sqrt{7}}{2}}, \quad
    \alpha = \frac{1}{2} \zeta^3 + \zeta, \quad
    d (\lambda) = \operatorname{Im} \bigl(\bar{\alpha} J_0 (\lambda \zeta) J_1 (\lambda \bar{\zeta})\bigr)
  \end{equation}
  where $J_0, J_1$ are the Bessel functions of the first kind. Fix $\lambda >
  0$ such that $d (\lambda) J_1 (\lambda) \neq 0$. Then the real-valued
  functions defined for all $r > 0$ by%
  \footnote{Note that the two determinations of the square root in the definition of~$\zeta$ yield the yield the same $A_{\lambda}$, $B_{\lambda}$ and~$C_\lambda$.}
  \begin{equation} \label{eq:ABCsol}
    A_{\lambda} (r) = \frac{2 \sqrt{2}}{d (\lambda)} \operatorname{Im} (J_1 (\lambda \bar{\zeta}) J_1 (\lambda \zeta r)), \quad
    B_{\lambda} (r) = 0, \quad
    C_{\lambda} (r) = \frac{1}{d (\lambda)} \operatorname{Im} (\bar{\alpha} J_1 (\lambda \bar{\zeta}) J_0 (\lambda \zeta r))
  \end{equation}
  are the unique solution of the differential
  system~\eqref{eq:EquationsForABCRepeat} satisfying the boundary conditions
  stated in Lemma~\ref{lem:EquationForABC}.
\end{lemma}

\begin{proof}
  Recall that, for $\nu = 0, 1$, Bessel's differential equation
  \begin{equation}
    x^2 y'' (x) + xy' (x) + (x^2 - \nu^2) y (x) = 0, \label{eq:Bessel}
  \end{equation}
  has a canonical basis of solutions consisting of the ordinary Bessel
  functions $J_{\nu} (x)$~and $Y_{\nu} (x)$ {\cite[§10.2]{DLMF}}. The
  function~$J_{\nu}$ is entire, while $Y_{\nu}$ is analytic on
  $\mathbb{C}\setminus\mathbb{R}_-$, and $Y_{\nu}(x)$ diverges to $-\infty$ as
  $x \rightarrow 0$ along the positive reals.

  Let $\mathcal{C}_{\nu}$ denote a cylinder function, i.e., a linear
  combination $aJ_{\nu} + bY_{\nu}$ with coefficients $a, b$ that do not
  depend on~$\nu$ {\cite[§10.2(ii)]{DLMF}}. One has~{\cite[(10.6.2),
  (10.6.3)]{DLMF}}
  \begin{equation} \label{eq:Bessel derivatives}
    \mathcal{C}_0^{'}(x) = - \mathcal{C}_1'(x), \qquad
    x \, \mathcal{C}_1'(x) + \mathcal{C}_1(x) = x \, \mathcal{C}_0(x).
  \end{equation}
  The equation for~$B_{\lambda}$ is exactly Bessel's equation with $\nu = 1$ and $x =
  \lambda r$; its general solution for $r > 0$ is hence $B_{\lambda} (r) =\mathcal{C}_1
  (\lambda r)$. Since $Y_1 (x)$ diverges as $x \rightarrow 0$ and $\lambda$~is
  nonzero, the initial condition at~$0$ forces $b = 0$. Similary, the initial
  condition at~$1$ implies $a = 0$ and therefore $B_{\lambda} = 0$, unless~$J_1
  (\lambda) = 0$, in which case any $B_{\lambda} (r) = aJ_1 (\lambda r)$ is a solution.

  Let us turn to the coupled equations for $A_{\lambda}$~and~$C_{\lambda}$. Make the ansatz
  \begin{equation} \label{eq:Bessel ansatz}
    C_{\lambda}(r) = f_0(r), \quad A_{\lambda}(r) = \alpha f_1(r), \qquad f_{\nu}(r)
    =\mathcal{C}_{\nu}(\lambda \zeta r),
  \end{equation}
  where $\alpha, \zeta$ are yet unspecified complex numbers. The equation
  involving~$C_{\lambda}''$ becomes
  \begin{equation} \label{eq:deq C transformed}
    rf_0''(r) + f_0'(r) + 2 \lambda^2 rf_0(r) + 2 \lambda \alpha(rf_1'(r)
    + f_1(r)) = 0.
  \end{equation}
  The change of variable passing from~$\mathcal{C}_0$ to~$f_0$ transforms
  Bessel's equation into
  \[ rf_0''(r) + f_0'(r) + \lambda^2 \zeta^2 rf_0(r) = 0, \]
  and the relations~{\eqref{eq:Bessel derivatives}} yield $rf_1' (r) + f_1 (r)
  = \lambda \zeta rf_0(r)$, so {\eqref{eq:deq C transformed}}~holds when
  \[ \lambda^2 \zeta^2 r = 2 \lambda^2 r + 2 \lambda \alpha \cdot \lambda
     \zeta r, \]
  i.e., when $\zeta^2 = 2 (1 + \alpha \zeta)$.

  Similarly, the equation involving~$A_{\lambda}''$ rewrites as
  \begin{equation} \label{eq:deq A transformed}
    r^2 \alpha f_1'' (r) + r \alpha f_1' (r) + (\lambda^2 r^2 - 1) \alpha
    f_1(r) + 2 \lambda r^2 f'_0 (r) = 0,
    \end{equation}
  and the last term on the left-hand side is equal to $2 \lambda^2 \zeta r^2
  f_1 (r)$ by~{\eqref{eq:Bessel derivatives}}. Thus, {\eqref{eq:deq A
  transformed}}~reduces to Bessel's equation provided that $\zeta^2 = 1 - 2
  \alpha^{- 1} \zeta$.

  In summary, the functions~{\eqref{eq:Bessel ansatz}} define a solution
  of~\eqref{eq:EquationsForABCRepeat} for any choice of $a, b$ in the
  definition of~$\mathcal{C}_{\nu}$ and $\alpha$, $\zeta$ such that
  $\zeta^2 = 2 (1 + \alpha \zeta) = 1 - 2 \alpha^{- 1} \zeta$.
  The latter condition is equivalent to
  \[ \zeta^4 + \zeta^2 + 2 = 0, \qquad \alpha = \zeta^3 / 2 + \zeta . \]
  Letting $\zeta$ now denote a fixed root of $\zeta^4 + \zeta^2 + 2$, say the
  one in~{\eqref{eq:zeta etc.}}, the choices
  \begin{equation} \label{eq:Bessel sols}
    C_{\lambda} (r) = J_0 (\lambda \zeta r), J_0 (\lambda \bar{\zeta} r), Y_0 (\lambda
    \zeta r), Y_0 (\lambda \bar{\zeta} r)
  \end{equation}
  provide us with four linearly independent\footnote{This follows, for
  instance, from the expressions {\cite[(10.2.2), (10.8.1)]{DLMF}} and the
  fact that $(\lambda \zeta)^2 \neq 1$.} solutions, which hence form a basis
  of the solution space of the system of two linear differential equation of
  order two.

  The asymptotic behaviour of~$Y_1$ at the origin,
  $Y_1 (\lambda \zeta r) \sim - 2 (\pi \lambda \zeta r)^{- 1}$~{\cite[(10.7.4)]{DLMF}},
  shows that linear combinations involving any of the last two
  solutions~{\eqref{eq:Bessel sols}} are incompatible with the
  conditions $A_{\lambda} (0) = C_{\lambda}' (0) = 0$. Therefore, one has
  \begin{align*}
    C_{\lambda} (r) &= uJ_0 (\lambda \zeta r) + vJ_0 (\lambda \bar{\zeta} r), \\
    A_{\lambda} (r) &= u \alpha J_1 (\lambda \zeta r) + v \bar{\alpha} J_1 (\lambda \bar{\zeta} r)
  \end{align*}
  for some $u, v \in \mathbb{C}$. The conditions $A_{\lambda} (1) = 0$, $C_{\lambda} (1) = 1$
  translate into a linear system for~$u, v$ of determinant
  \[ \bar{\alpha} J_0 (\lambda \zeta) J_1 (\lambda \bar{\zeta}) - \alpha J_1
     (\lambda \zeta) J_0 (\lambda \bar{\zeta}) = 2 i d (\lambda) \]
  (where we have used the fact that $J_{\nu} (\bar{z}) = \overline{J_{\nu}
  (z)}$ {\cite[(10.11.9)]{DLMF}}). When $d (\lambda) \neq 0$, the unique
  solution is $u = - \bar{v} = \bar{\alpha} J_1 (\lambda \bar{\zeta})$. Since
  $| \alpha |^2 = 2 \sqrt{2}$, this leads to the expressions~\eqref{eq:ABCsol}.
\end{proof}

\section{Concluding}
\label{sec:concluding}

\begin{lemma}
  \label{lem:ExistenceOfZero}
  In the notation of Lemma~\ref{lem:SolvingForABC}, there exists
  $\tilde\lambda > 0$ such that $C(0)$, viewed as a function of~$\lambda$,
  has a pole at~$\tilde\lambda$.
\end{lemma}

\begin{proof}
  Let us first show that $d(\lambda)$ has a zero lying in the
  interval~$(2.5, 3)$.
  Consider the series expansions~{\cite[(10.2.2)]{DLMF}}
  \begin{equation}
    \label{eq:Bessel series} J_0 (x) = \sum_{k = 0}^{\infty} (- 1)^k  \frac{(x
    / 2)^{2 k}}{k!^2}, \qquad J_1 (x) = \sum_{k = 0}^{\infty} (- 1)^k
    \frac{(x / 2)^{2 k + 1}}{k! (k + 1) !} .
  \end{equation}
  For $x \in \mathbb{C}$ and $n \in \mathbb{N}$ such that $| x | < 2 (n + 1)$,
  the remainders starting at index~$n$ of both series are bounded by
  \begin{equation}
    \sum_{k = n}^{\infty} \frac{| x / 2 |^{2 k}}{k!^2} = \frac{| x / 2 |^{2
    n}}{n!^2}  \sum_{k = n}^{\infty} \frac{| x / 2 |^{2 (k-n)}}{((n + 1)
    \cdots (n + k))^2} \leqslant \frac{1}{1 - | x |^2 / (2 n + 2)^2}  \frac{| x /
    2 |^{2 n}}{n!^2} . \label{eq:tail bound}
  \end{equation}
  In particular, for $x = \lambda \zeta$ or $x = \lambda \bar{\zeta}$ with $0
  < \lambda \leqslant 3$, we have $| x / 2 | < 1.784$. For $n = 5$, the
  quantity~\eqref{eq:tail bound} is bounded by~$0.025$. By replacing
  $J_0$~and~$J_1$ by the first five terms of the series~\eqref{eq:Bessel series}
  in the expression of~$d(\lambda)$ and propagating this bound by the triangle
  inequality, one can check that $d (2.5) < - 0.06$. A similar calculation
  shows that $d (3) > 0.03$. Since $d (\lambda)$ is a continuous function
  of~$\lambda$, it follows that $d (\tilde{\lambda})$ vanishes for
  some~$\tilde{\lambda} \in (2.5, 3)$.

  We still need to check that the numerator of $C_{\lambda} (0)$ in~\eqref{eq:ABCsol} does
  not vanish at~$\tilde{\lambda}$. One has $J_0 (0) = 1$. Taking $n = 3$
  in~\eqref{eq:tail bound} yields an expression of the form
  \[ \operatorname{Im} (\bar{\alpha} J_1 (\lambda \bar\zeta)) = \operatorname{Im} (c_0 \lambda +
     c_1 \lambda^3 + c_2 \lambda^5 + \bar{\alpha} b),
     \qquad | b | \leqslant 1.12 \]
  where one can check that $\operatorname{Im} (c_0) < - 0.33$, $\operatorname{Im} (c_1) < -
  0.12$, $\operatorname{Im} (c_2) < -0.001$. For all
  $\lambda \geqslant 2.5$, this implies
  \[ \operatorname{Im} (\bar{\alpha} J_1 (\lambda \zeta)) \leqslant \operatorname{Im} (c_0)
     \lambda + \operatorname{Im} (c_1) \lambda^3 + | \alpha |  | b | \leqslant - 1.3.
  \]
  The claim follows.
\end{proof}

\begin{remark}
  Instead of doing the calculation sketched in the proof manually, one can
  easily prove the result using a computer implementation of Bessel functions
  that provides rigorous error bounds. For example, using the interval
  arithmetic library Arb~{\cite{Johansson2017}} via SageMath, the check that
  $d (\lambda)$~has a zero goes as follows. The quantities of the form
  {\texttt{[x.xxx +/- eps]}} appearing in the output are guaranteed to be
  rigorous enclosures of the corresponding real quantities. We check the
  presence of a zero in the interval $[2.82, 2.83]$ instead of $[2.5, 3.0]$
  because having a tighter estimate simplifies the second step.
  \begin{verbatim}
  sage: zeta = CBF(sqrt((-1+I*sqrt(7))/2))
  sage: alpha = zeta^3/2 + zeta
  sage: lb, ub = CBF(282/100), CBF(283/100)
  sage: (alpha.conjugate()*(lb*zeta).bessel_J(0)
  ....:                   *(lb*zeta.conjugate()).bessel_J(1))
  [-13.208370024264 +/- 4.16e-13] + [-0.003639973760 +/- 4.63e-13]*I
  sage: (alpha.conjugate()*(ub*zeta).bessel_J(0)
  ....:                   *(ub*zeta.conjugate()).bessel_J(1))
  [-13.424373315124 +/- 4.75e-13] + [0.005782411521 +/- 4.38e-13]*I
  \end{verbatim}
  One can then verify as follows that the image by the function $\lambda \mapsto
  \bar{\alpha} J_1 (\lambda \bar{\zeta})$ of the interval $[2.82, 2.83]$
  only contains elements of negative imaginary part.
  \begin{verbatim}
  sage: crit = lb.union(ub); crit # convex hull (real interval)
  [2.8 +/- 0.0301]
  sage: alpha.conjugate()*(crit*zeta.conjugate()).bessel_J(1)
  [+/- 0.0707] + [-4e+0 +/- 0.303]*I
  \end{verbatim}
\end{remark}

\begin{theorem}
  The series expansion with respect to~$\lambda$ of $\Phi(0)$
  has a finite radius of convergence.
\end{theorem}

\begin{remark}
  Since the condition of \cite{ChevyrevLyons} of uniqueness of laws is only \emph{sufficient},
  the questions remains on whether there exists another law on $G$
  having the same moments as $S(X^0)_{0,T}$.
\end{remark}
\begin{proof}
Assume for contradiction that $\Phi(0)$ has an infinite
radius of convergence. Then $F_\lambda(0)$ is an entire function
in $\lambda$. We also know from Corollary \ref{cor:A and C}
that there exists $\lambda^{*}>0$ such that for real $\lambda<\lambda^{*}$
\begin{equation}
F_\lambda(0)=\begin{pmatrix}
\cos\theta & -\sin\theta & 0\\
\sin\theta & \cos\theta & 0\\
0 & 0 & 1
\end{pmatrix}\begin{pmatrix}
A_{\lambda}(0)\\
B_{\lambda}(0)\\
C_{\lambda}(0)
\end{pmatrix},\label{eq:PolarDecompositionRecalled}
\end{equation}
where $A_{\lambda},B_{\lambda},C_{\lambda}$ are defined by Lemma \ref{lem:SolvingForABC}.
By the Identity theorem, $A_{\lambda},B_{\lambda},C_{\lambda}$ are entire functions and
\eqref{eq:PolarDecompositionRecalled} holds
for all~$\lambda$.
This contradicts Lemma~\ref{lem:ExistenceOfZero}, and therefore $\Phi(0)$
has a finite radius of convergence.
\end{proof}

\section{Appendix}
\label{sec:appendix}
Let $\Gamma$ be a domain in $\mathbb{R}^{d}$.
\begin{definition}
Let $u$ be a locally integrable function in $\Gamma$ and $\alpha$
be a multi-index. Then a locally integrable function $r_{\alpha}u$
such that for every $g\in C_{c}^{\infty}(\Gamma)$, 
\begin{eqnarray*}
\int_{\Gamma}g(x)r_{\alpha}(x)dx=(-1)^{\vert\alpha\vert}\int_{\Gamma}D^{\alpha}g(x)u(x)dx,
\end{eqnarray*}
will be called \emph{weak derivative} of $u$ and $r_{\alpha}$ is denoted
by $D^{\alpha}u$. By convention, $D^{\alpha}u=u$ if $\vert\alpha\vert=0$. 
\end{definition}

\begin{definition}
  The \emph{Sobolev space} $W^{k,p}(\Gamma)$ for $p, k \in \mathbb N$ is defined to be the set of all $\mathbb{R}^{\tilde{d}}$-valued
functions $u\in L^{p}(\Gamma)$ such that for every multi-index $\alpha$
with $\vert\alpha\vert\leq k$, the weak partial derivative $D^{\alpha}u$
belongs to $L^{p}(\Gamma)$, i.e. 
\begin{eqnarray*}
W^{k,p}(\Gamma)=\left\{ u\in L^{p}(\Gamma):D^{\alpha}u\in L^{p}(\Gamma)\,\,\forall|\alpha|\leq k\right\} .
\end{eqnarray*}
It is endowed with the Sobolev norm defined as follows: 
\begin{eqnarray*}
\vert\vert u\vert\vert_{W^{k,p}(\Gamma)}=\sum_{j=1}^{\tilde{d}}\left(\sum_{\vert\alpha\vert\leq k}\int_{\Gamma}\vert D^{\alpha}u^{j}(x)\vert^{p}dx\right)^{1/p}.
\end{eqnarray*}
When $k=0$, this norm coincides with the $L^{p}(\Gamma)$-norm,
i.e. 
\begin{eqnarray*}
\vert\vert u\vert\vert_{W^{k,p}(\Gamma)}=\vert\vert u\vert\vert_{L^{p}(\Gamma)}.
\end{eqnarray*}
\end{definition}

\begin{theorem}
\label{boundaryRegularity} Let $M$ be a second order differential
operator with coefficients $\{a^{i,j}\}$. Let $u$ be a weak solution
of 
\begin{eqnarray*}
Mu=f(x),\\
u-g\in H_{0}^{1,2}(\Gamma).
\end{eqnarray*}
Suppose that the ellipticity condition holds. Let $f\in W^{k,2}(\Gamma)$,
$g\in W^{k+2,2}(\Gamma)$. Let $\Gamma$ be a bounded domain of class $C^{k+2}$ and
let the coefficients of $M$ be of class $C^{k+1}(\bar{\Gamma})$.
Then 
\begin{eqnarray*}
\vert\vert u\vert\vert_{W^{k+2,2}(\Gamma)}\leq c\left(\vert\vert f\vert\vert_{W^{k,2}(\Gamma)}+\vert\vert g\vert\vert_{W^{k+2,2}(\Gamma)}\right),
\end{eqnarray*}
with $c$ depending on $\lambda,d$, $\Gamma$ and on the $C^{k+1}$-norms
for the $a^{i,j}$.
\end{theorem}
\begin{proof}
It is proved by using Theorem 8.13 in \cite{gilbarg2015elliptic} and setting the boundary condition $\varphi = 0$.
\end{proof}
In the following we prove Lemma \ref{lemma_pde} for $m\geq2$, which
is a generalization of Lemma 3.11 for the case $m=\lfloor\frac{d}{2}\rfloor$
in \cite{LyonsNi}. 
\begin{lemma}
\label{lemma_pde} Let $\Gamma$ be a bounded domain of class $C^{m}$
in $\mathbb{R}^{d}$, where $m\geq2$. Then there exists a constant
$C$ only depending on $\Gamma$ and $d$, such that for every positive
integer $n\geq2$, 
\begin{eqnarray}
\vert\vert\proj_{n}(\Phi)\vert\vert_{W^{m,2}(\Gamma)}\leq C\left(\vert\vert\proj_{n-1}(\Phi)\vert\vert_{W^{m,2}(\Gamma)}+\vert\vert\proj_{n-2}(\Phi)\vert\vert_{W^{m,2}(\Gamma)}\right).\label{inequality1}
\end{eqnarray}
\end{lemma}
\begin{proof}
The proof of Lemma 3.11 in \cite{LyonsNi} can be applied here directly,
except for that we need to check that $\proj_{n}(\Phi)\in W^{m,2}$,
which is proved in the following theorem \ref{Phi_theorem}. 
\end{proof}
\begin{theorem}
\label{Phi_theorem} Suppose that $\Gamma$ is a non-empty bounded
domain in $E$. It follows $\Phi$ is infinitely differentiable in
componentwise sense, i.e. for all index $I$, $\proj_{n}\circ\Phi$
is infinitely differentiable for all $n$.
\end{theorem}
\begin{proof}
Based on Theorem 3.2 in \cite{LyonsNi}, it shows that 
\[
\Phi(z)=\int_{\Gamma}G_{\varepsilon}(z-y)\otimes\Phi_{\Gamma}(y)dy=G_{\varepsilon}*\Phi(z),
\]
where $K_{\varepsilon}(r)$ is a smooth distribution with compact
support $[0,\frac{\varepsilon}{2}]$, $*$ is the convolution, and
$G_{\varepsilon}$ be a map from $\mathbb{R}^{d}$ to $T((\mathbb{R}^{d}))$
defined by: 
\begin{eqnarray*}
 &  & \Psi(z)=\Psi(z)=\mathbb{E}^{0}\left[S(B_{\left[0,\tau_{\mathbb{D}(0,|z|)}\right]})|B_{\tau_{\mathbb{D}(0,|z|)}}=z\right].\\
 &  & G_{\varepsilon}(z)=\Psi(-z)K_{\varepsilon}(|z|).
\end{eqnarray*}
Since $\Psi$ is smooth (in polynomial form) and $K_{\varepsilon}$
is a smooth function with compact support, $G_{\varepsilon}$ is a
smooth function with compact support. It is easy to show that for
any partial derivative $D^{\alpha}G_{\varepsilon}$ is $L_{1}$ integrable.
\[
||D^{\alpha}G_{\varepsilon}||_{L^{1}}<+\infty.
\]
On the other hand, $\Phi\in L^{1}$ as well, and so we have 
\[
||G_{\varepsilon}*D^{\alpha}\Phi||_{L^{1}}<+\infty.
\]
Thus $G_{\varepsilon}*\Phi$ is infinitely differentiable, since $D^{\alpha}\Phi=(D^{\alpha}G_{\varepsilon})*\Phi\in L^{1}$.
\end{proof}

\printbibliography

\end{document}